\documentclass[12pt,amstex]{article}
\usepackage[usenames]{color}

\usepackage{amscd}
\usepackage{amsmath}
\usepackage{amssymb}
\usepackage{graphicx}

\newtheorem {theo} {\bf Theorem} [section]
\newtheorem {prop} [theo] {\bf Proposition}

\newtheorem {lem} [theo] {\bf Lemma}
\newtheorem {defn} [theo] {\bf Definition}

\newtheorem {rem} [theo] {\bf Remark}
\newcommand{\QED}{\hfill \CaixaPreta \vspace{6mm}}
\def\CaixaPreta{\vrule Depth0pt height6pt width6pt}

\newcommand{\qed}{\nopagebreak\hfill{\vrule width6pt height6pt depth0pt}}
\newcommand{\be}{\begin{eqnarray}}
\newcommand{\ee}{\end{eqnarray}}
\newcommand{\benn}{\begin{eqnarray*}}
\newcommand{\eenn}{\end{eqnarray*}}
\newcommand{\bse}{\begin{equation}}
\newcommand{\ese}{\end{equation}}
\newcommand{\bsenn}{\begin{displaymath}}
\newcommand{\esenn}{\end{displaymath}}
\newcommand{\N}{\mathbb{N}}

\newcommand{\R}{\mathbb{R}}

\usepackage{graphicx}

\begin{document}

\title{Complete Expanding Interval Maps are Mixing}
\author{E.L. O'Neil, F.M. Tangerman, J.J.P. Veerman}
\maketitle

\noindent
\section*{Abstract}
We give a new and elementary\footnote{That is: without functional analysis.} proof that certain smooth, complete, expanding maps from the interval to itself
(Definition \ref{def:mix-CEIM}) are mixing.

\vskip 0.1in\noindent
KEYWORDS: Ergodicity, mixing, bounded nonlinearity.

%\vskip 0.4in\noindent
\begin{centering}
\section{Introduction}
\label{chap:intro}
\end{centering}
\setcounter{figure}{0} \setcounter{equation}{0}

The aim of this paper is to give an elementary proof that maps like the Gauss map,
$T:x\rightarrow 1/x-\lfloor 1/x\rfloor$, are mixing (Theorem \ref{thm:mix-mixing}). It is
known that this map is \emph{exponentially} mixing \cite{billingsley, CFS, viana-oli, zwei},
however the proofs are far from elementary. Instead, we give a comparatively simple proof of mixing.

In this introduction we set the scene, while the next section contains the proofs. First we fix
some notation. For an interval
map $T$, we write $\partial T(x)$ to indicate the derivative with respect to the argument at $x$, while
$\partial_sT(x)$ will indicate $T'(x)\frac{dx}{ds}$.

\begin{defn}[CEIM] Let $T:[0,1]\rightarrow [0,1]$ be $C^2$. We say $T$ is a complete expanding interval
map (CEIM, for short) if\\
(i) there are $C>0$ and $\Delta>1$ so that for all $x$,
\bse
|\partial T^n(x)|\geq C\Delta^n \,.
\label{eq:mix-expanding}
\ese
(ii) The unit interval can be partitioned into at most countably many (and at least two) intervals $I_i$
of positive length such that $T$ restricted to $I_i$ is onto $[0,1]$.\\
(iii) The nonlinearity of $T$ is bounded, that is: there is $M>0$ such that
\bsenn
\sup_{i}\sup_{z,u\in I_i}\,\left|\frac{\partial^2T(z)}{\partial T(z)\;\partial
T(u)}\right| \leq M\,,
\esenn
where the supremum is over the intervals defined in (ii).
\label{def:mix-CEIM}
\end{defn}

\noindent
The intervals $I_i$ will be called the first level fundamental
intervals. Any subsequent iterate of $T$ will then induce a finer partition of $[0,1]$. The intervals
on which the restriction of $T^2$ is a bijection onto $[0,1]$ will be called the second level
fundamental intervals, and so forth.

It is useful to recall that for a CEIM $T$, the nonlinearity of $T^n$ is also bounded: there is $D>0$
such that for all $n>0$:
\bse
\sup_{i}\sup_{z,u\in I_i}\,\left|\frac{\partial^2T^n(z)}{\partial T^n(z)\;\partial
T^n(u)}\right| \leq D \,,
\label{eq:distortion1}
\ese
where now $I_i$ ranges over the $n$th level fundamental intervals (\cite{veer}, Proposition 10.16). In
addition, for all $n$ and $x$ in  an $n$th level fundamental interval $I_i$:
\bse
D^{-1} \leq |\partial T^n(x)|\,|I_i|\leq D \,.
\label{eq:distortion2}
\ese

\begin{figure}[!ht]
\centering
\includegraphics[width=2.5in]{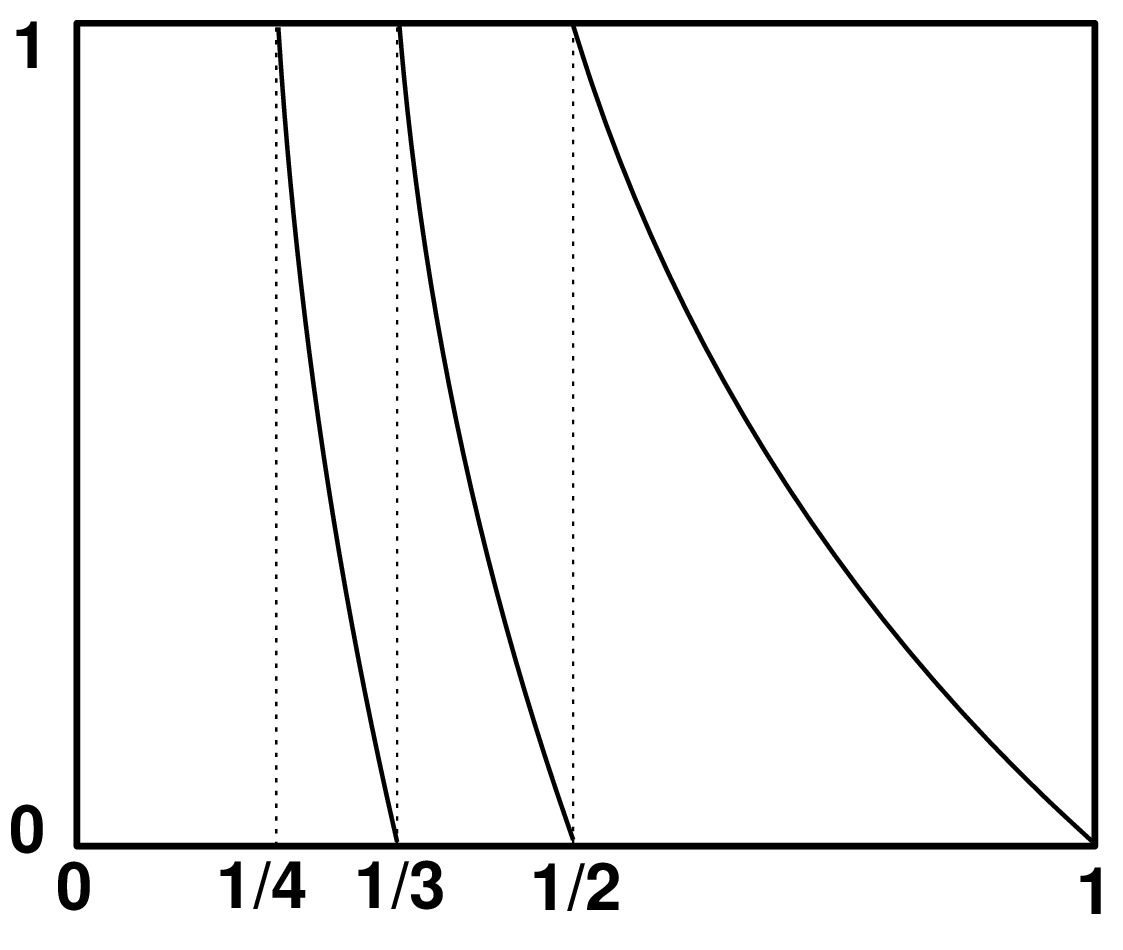}
\caption{\emph{A sketch of three branches of the Gauss map.}}
\label{fig:mix-Gauss}
\end{figure}

Clearly, for any integer $n>1$, the map $x\rightarrow nx \mod 1$ satisfies all requirements (its second
derivative being zero). It is an easy, but tedious, exercise to show that for $d\geq 1$, the map
$x\rightarrow x^{-d} \mod 1$ also satisfies all requirements. We leave this to the reader. Because of
its prominent role in number theory, the most important of these examples is of course the Gauss map
of which a few branches are sketched in Figure \ref{fig:mix-Gauss}.

\vskip -0.0in
It is well-known that a CEIM $T$ is ergodic with respect to an invariant measure $m$ that is absolutely
continuous with respect to the Lebesgue measure. We lift the following from \cite{veer}, appendix E,
where a proof can also be found.

\vskip 0.0in\noindent
\begin{theo} Let $T$ be a complete expanding interval map. Then there is a unique invariant
probability measure $m$ that has the following properties:\\
(i) $m$ has a continuous density $\mu$,\\
(ii) there is a $K>1$ such that for all $x\in [0,1]$, $K^{-1}\leq \mu(x) \leq K$,\\
(iii) $T$ is ergodic with respect to $m$.
\label{thm:mix-ergodic}
\end{theo}

Consider the following class of functions.

\vskip -0.1in\noindent
\begin{defn} A family of functions $\{f_n\}_{n=0}^\infty=\{j_n+h_n\}_{n=0}^\infty$ from $[0,1]$ to
$\R$ is called piecewise equicontinuous if:\\
(i) The functions $j_n$ are constant except for jump discontinuities\footnote{A jump discontinuity of $j$
at $x$ means that the right and left limit of $j$ at $x$ exist and are finite, but not equal.};
$j_n$ has bounded variation $J_n$ and $\lim J_n=0$.\\
(ii) The functions $h_n$ are uniformly equicontinuous, that is: for all $\epsilon>0$, there is
$\delta>0$ such that for all $h_n$, $|y-x|<\delta$ implies
$|h_n(y)-h_n(x)|<\epsilon$.
\label{def:mix-densityclass}
\end{defn}

Our main result is Theorem \ref{thm:mix-mixing}. It states that CEIM's are mixing. Here is the
definition.

\vskip -0.1in \noindent
\begin{defn} Let $T$ be a transformation with an invariant measure $m$. Then $T$ is mixing with respect
to $m$ if for any two intervals $A$ and $B$
\bsenn
\lim_{n\rightarrow \infty}|m(T^{-n}(A)\cap B)-m(A)\cdot m(B)| =0 \,.
\esenn
\label{defn:mix-mixing}
\end{defn}

%\vskip 0.4in\noindent
\begin{centering}
\section{Proofs}
\label{chap:mix}
\end{centering}
\setcounter{figure}{0} \setcounter{equation}{0}

\vskip -0.2in\noindent
\begin{prop} Let $\{f_n\}_{n=0}^\infty=\{j_n+h_n\}_{n=0}^\infty$ a piecewise equicontinuous family.
Then the sequence $\{f_{n}\}$ has a subsequence that converges uniformly to a continuous function $g$.
\label{thm:mix-arzela}
\end{prop}

\vskip 0.0in\noindent
{\bf Proof.} The proof follows largely that of the Arzel\`a-Ascoli theorem. Choose $\epsilon>0$ and
let $\delta$ be such that the family $h_n$ is uniformly $\epsilon$-$\delta$ equicontinuous.
Furthermore, let $X$ be a dense set from which we can choose finitely many $x_i$ a distance less
than $\delta$ apart.

Since the $f_n$ are uniformly bounded, there is a subsequence $f_{1,n}$ of $f_n$ such that
$f_{1,n}(x_1)$ converges to $y_1$. Then we can choose a subsequence $f_{2,n}$ of $f_{1,n}$ such that
$f_{2,n}(x_2)$ converges to $y_2$. And so forth.
Now rename $g_n:=f_{n,n}=f_{k_n}$ for some $k_n\geq n$. Clearly,
\bsenn
\lim_{n\rightarrow \infty} g_n(x_k)=y_k \,.
\esenn
Now for any $x$, take an $x_i$ less than $\delta$ away from $x$. Since $g_n=f_{k_n}$, the
jumps in $g_n$ have variation at most $J_{k_n}<J_n$. Thus $|g_n(x_i)-g_n(x)|<\epsilon+J_{n}$. Also,
$\{g_n(x_i)\}_{n=1}^\infty$ is Cauchy, and thus for $\ell>n$ and $n$ large
$|g_\ell(x_i)-g_n(x_i)|<\epsilon$. Therefore,
\benn
|g_\ell(x)-g_n(x)|&\leq& |g_\ell(x)-g_\ell(x_i)|+|g_\ell(x_i)-g_n(x_i)|+|g_n(x_i)-g_n(x)|\\
&\leq& (\epsilon+J_{n})+\epsilon+(\epsilon+J_{n})< 4\epsilon \,,
\eenn
for $n$ large enough. Thus $g_n(x)$ is Cauchy and uniformly converges to some function $g(x)$.

In fact, $g$ is uniformly continuous, because if $|y-x|<\delta$, then
\benn
|g(y)-g(x)|&\leq& |g(y)-g_n(y)|+|g_n(y)-g_n(x)|+|g_n(x)-g(x)|\\
&\leq& 4\epsilon+(\epsilon+J_{n})+4\epsilon <10\epsilon \,,
\eenn
for $n$ large enough. \QED

Before we prove the next results, we need more notation. The pushforward (or
Ruelle-Perron-Frobenius operator) of a density $\phi_0$ under $T$ is $\phi_1$:
\bsenn
\phi_1(y)=(T_*\phi_0)(y) := \sum_{x\in T^{-1}y} \frac{\phi_0(x)}{|\partial T(x)|} \,,
\esenn
and $\phi_n=T_*\phi_{n-1}$. By Theorem \ref{thm:mix-ergodic}, $T$ is ergodic with respect to
an invariant density $\mu$, which is continuous and bounded from above and below. Define
$f_n(x):=\phi_n(x)/\mu(x)$. We have
\bsenn
f_{n}(y)= \sum_{x\in T^{-1}(y)}f_{n-1}(x)\,\,\frac{\mu(x)}{\mu(T(x))\,|\partial T(x)|} =
\sum_{x\in T^{-n}(y)}f_{0}(x)\,\,\frac{\mu(x)}{\mu(T^n(x))\,|\partial T^n(x)|}\,.
\esenn
By invariance, $1=\frac{1}{\mu(y)}\,\sum\frac{\mu(x)}{|\partial T^n(x)|}$, and thus these operations
are \emph{weighted averages}. The weights associated with $T^n$ (for $n\in \N$) are positive and will
be denoted by $w_n(x)$. Thus
\bse
f_{n}(y) = \sum_{x\in T^{-n}(y)}f_0(x)w_n(x)\,.
\label{eq:weighted}
\ese

The following estimate will be necessary to obtain a modulus of continuity on $f_n$ independent of $n$.

\begin{lem}
There is an increasing real-valued function $\epsilon_w$ with $\lim_{t\to 0}\epsilon_w(t)=0$
such that for all $y$ and $y'$, $n>0$, and $x_i,x_i'$ in the same $n$th level fundamental interval
with $T^n(x_i)=y$ and $T^n(x_i')=y'$,
\bsenn
\sum_i|w_n(x_i)-w_n(x_i')|\leq\epsilon_w(|y-y'|)\,.
\esenn
\vskip -0.2in\noindent
(That is: $\epsilon_w$ is a modulus of continuity for $w_n$ for all $n$.)
\label{lem:weights-ratio}
\end{lem}

\vskip -0.1in\noindent
{\bf Proof.} First, for notational convenience we define $a:=1/\mu(y)=1/\mu(T^n(x_i))$, $b:=\mu(x_i)$, and
$c:=1/|\partial T^n(x_i)|$, and similarly, $a':=1/\mu(y')=1/\mu(T^n(x_i'))$, et cetera. Taking the sum over
all $n$th-level fundamental intervals, we have
\bsenn
\sum_i|abc-a'b'c'|\leq \sum_i|(a-a')bc|+\sum_i|(b-b')a'c|+\sum_i|(c-c')a'b'| \,.
\esenn
Now, $a$ and $b$ are uniformly bounded (recall that $\mu$ is bounded away from zero) and the sum $\sum_i |c|$
is bounded by $\sum_i D|I_i|=D$. In the first term on the right hand side $|a-a'|$ has a modulus of continuity
inherited from $\mu(y)$. In the second term, we have a weighted sum of moduli $\epsilon_b(|x_i-x_i'|)$ of
$b$. Since $|x_i-x_i'|\leq |y-y'|$, we assert that $\epsilon_b(|y-y'|)$ serves as modulus for this term.

We need to find a modulus for the third term. Rewrite it as $\sum_i|(c/c'-1)a'b'c'|$.
It now is a weighted average of $|c/c'-1|$ which, in turn, we estimate using \eqref{eq:distortion1}
and \eqref{eq:distortion2}.
\bsenn
\left|\ln\frac{|\partial T^n(x)|}{|\partial T^n(x')|}\right|=|\ln|\partial T^n(x)|-\ln|\partial T^n(x')||=\left|\int_{x'}^x \frac{\partial^2 T^n(z)}{(\partial T^n(z))^2}\cdot\partial T^n(z)dz\right|\leq D|y-y'| \,.
\esenn
\bsenn
\Longrightarrow\quad\quad \quad e^{-D|y-y'|}\leq \frac{|\partial T^n(x)|}{|\partial T^n(x')|} \leq e^{D|y-y'|}\,.\quad \quad\quad \quad\quad \quad\quad \quad
 \esenn
And this imparts a modulus of continuity $\epsilon_C(|y-y'|)$ on the third term.
\QED

In what follows, we will decompose $f_n$ into $j_n+h_n$, jumps plus a continuous function. To facilitate
the proof of the next lemma, it is helpful to preface it with an exercise in bookkeeping. Consider the
situation if in Figure \ref{fig:mix-disconts}. We calculate the continuous part of the difference
$f(x)w(x)-f(x')w(x')$, where $w(x)$ is continuous in the interval $[x,x']$ but $f$ has two jump
discontinuities:
\bsenn
\left[f(x)w(x)-f(x^{1-})w(x^{1-})\right]+\left[f(x^{1+})w(x^{1+})-f(x^{2-})w(x^{2-})\right]+\left[f(x^{2+})w(x^{2+})-f(x')w(x') \right]\,.
\esenn
Now take a limit as $x^{j+}$ and $x^{j-}$ tend to $x^j$, the locus of the discontinuity, and denote
the jump in $f$ at $x^j$ by $J^j$. We use continuity of $w(x)$ to get
\benn
\cdots &=& f(x)w(x) + \sum_j J^jw(x^j)-f(x')w(x') \\
   &=& \left[f(x)-(f(x')-\sum_jJ^j)\right]w(x) + f(x')(w(x)-w(x'))+\sum_jJ^j(w(x^j)-w(x)) ,.
\eenn
Finally, we note that $f(x)-(f(x')-\sum_jJ^j)$ can be replaced by $h(x)-h(x')$ if $h$ is the continuous
part of $f$.

\begin{figure}[!ht]
\centering
\includegraphics[width=3.0in]{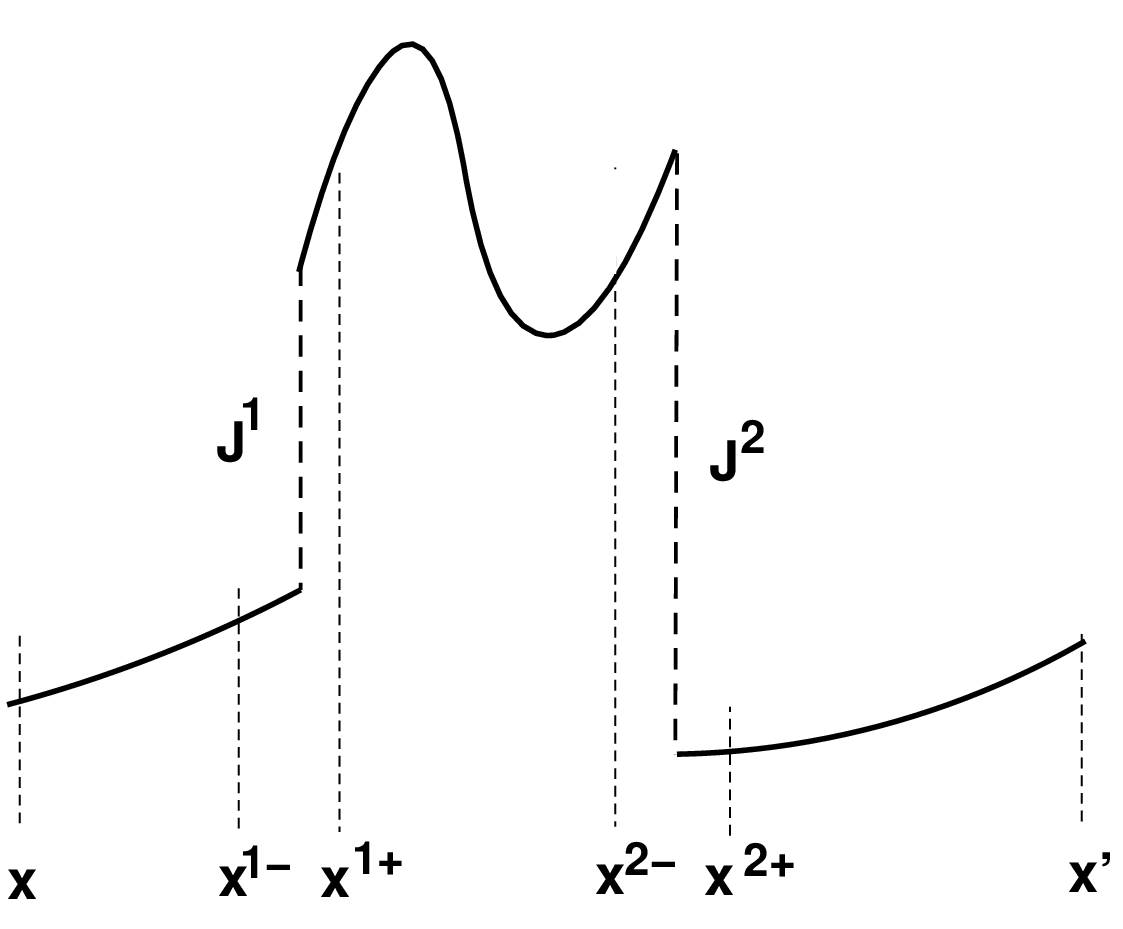}
\caption{\emph{Sketch of a function $f$ that is the sum of a continuous function $h$ and two jumps $J^j$.}}
\label{fig:mix-disconts}
\end{figure}

\vskip -0.2in\noindent
\begin{prop} Let $j_0$ be piecewise constant with finitely many finite jumps and $h_0$ continuous.
Then $h_n$ defined in \eqref{eq:weighted} has a modulus of continuity which is independent of $n$.
\label{prop:mix-modulus}
\end{prop}

\vskip 0.0in\noindent
{\bf Proof.} By Equation \eqref{eq:weighted}, we have
\bsenn
|f_n(y)-f_n(y')| = \left|\sum_{i}f_0(x_i)w_n(x_i)-f_0(x_i')w_n(x_i')\right| \,,
\esenn
where $\sum_{i}$ indicates, as before, that $x_i'\in T^{-n}(y')$ and $x_i\in T^{-n}(y)$ are in the same $n$th
level fundamental interval $I_i$. We wish to obtain a modulus of continuity for the continuous part $h_n$ of
$f_n$ that is independent of $n$. Denote the jumps of $f_0$ in the interval $[x_i,x_i']$ by $J_i^j$ and
the loci of the jumps by $x_i^j$.  By the observations prior to this lemma, we obtain
\bsenn
h_n(y)-h_n(y') = \sum_{i}(h_0(x_i)-h_0(x_i'))w_n(x_i) +\sum_{i}f_0(x_i')(w_n(x_i)-w_n(x_i'))+
\sum_{i}\sum_{j}J_i^j(w_n(x_i^j)-w_n(x_i)) \,.
\esenn
We now give a modulus for each of the terms on the right hand side. The first term consists of a weighted
sum of moduli $\epsilon_0(|x_i-x_i'|)$ of $h_0$. Thus $\epsilon_0(|y-y'|)$ serves as modulus for this term.
The second term has $\epsilon(t)$ of Lemma \ref{lem:weights-ratio} multiplied by $||f_0||$ as a modulus.
The absolute value of the third term is bounded by the variation of $j_0$ times the modulus
computed in Lemma \ref{lem:weights-ratio}. \QED

\vskip -0.3in\noindent
\begin{prop} If $\phi$ is the sum of a continuous function and a finite jump function and
$\int \phi=0$, then $\lim_{n\rightarrow \infty}||T_*^n\phi||=0$, where $||.||$ denotes the supnorm.
\label{prop:mix-decay}
\end{prop}

\vskip 0.0in\noindent
{\bf Proof.} Denote the operator introduced in \eqref{eq:weighted} by $\widetilde T_*^n$.
Since it is a weighted average, $||f_{n+1}||=||\widetilde T_*f_n||\leq ||f_n||$. Thus
$||f_n||$ is positive and decreasing and so $\lim_{n\rightarrow \infty} ||f_n||$ exists.
Call that limit $L$. We assume $L>0$ and derive a contradiction.

Consider \eqref{eq:weighted} and decompose $f_n$ as before into $j_n+h_n$.
The derivatives $|\partial T^n(x)|$ increase exponentially, and therefore the jumps in $j_n$ decrease
exponentially. The $h_n$ are uniformly equicontinuous by Proposition \ref{prop:mix-modulus}. Thus by
Proposition \ref{thm:mix-arzela}, the $f_n$ have a subsequence $f_{n_i}$ uniformly convergent to a
continuous function $g$.

By assumption $||g||=L>0$. But since $g$ has average zero,
$g(x)$ must alternate sign, and therefore by continuity it must assume the value zero at some point.
Again by continuity, there is an interval around that point where $|g(x)|<L/2$. Some $m$th iterate
of $\widetilde T_*$ must sample $g$ in precisely that interval. But then there is a $\eta>0$ so that
\bsenn
||\widetilde T_*^mg||\leq L-2\eta \,.
\esenn
Now by linearity,
\bsenn
||\widetilde T_*^mf_{n_i}||=||\widetilde T_*^mg+\widetilde T_*^m(f_{n_i}-g)||\leq ||\widetilde T_*^mg||+
||\widetilde T_*^m(f_{n_i}-g)||\leq L-2\eta + ||\widetilde T_*^m(f_{n_i}-g)||\,.
\esenn
But we can take $n$ large enough so that $||f_{n_i}-g||<\eta$. Since $\widetilde T_*$ cannot increase
the supnorm, we finally obtain $||\widetilde T_*^mf_{n_i}||\leq L-\eta$, contradicting that
$||f_{n_i}||\geq L>0$. \QED

\vskip -0.2in \noindent
\begin{theo}[CEIM's are mixing] Let $T$ be a CEIM with invariant measure $m$, absolutely continuous
with respect to Lebesgue. Then $T$ is mixing with respect to $m$.
\label{thm:mix-mixing}
\end{theo}

\vskip -0.1in\noindent
{\bf Proof.}
By Theorem \ref{thm:mix-ergodic}, $T$ has invariant density $\mu$ and invariant measure $m$. Given an
interval $A$, ${\bf 1}_A$ is the characteristic function on $A$ with $m$-length $m_A$. Then
\benn
\left| m(T^{-n}(A)\cap B)-m_A m_B\right|&=&\left|\int_{T^{-n}(A)\cap B}\mu(y)dy-\int_B m_A
\mu(y)dy\right|\\
&=& \left|\int_{B}T_*^{n}({\bf 1}_A\cdot \mu)(y)dy-\int_B m_A \mu(y)dy\right| \,.
\eenn
Since $\mu$ is invariant under $T_*$, we may replace the integrand by $(T_*^n[{\bf 1}_A\mu-m_A\mu])(y)$.
Thus
\bsenn
\left| m(T^{-n}(A)\cap B)-m_A m_B\right| \leq \int |(T_*^n[{\bf 1}_A\mu-m_A\mu])(y)|\,dy \,.
\esenn
By Theorem \ref{thm:mix-ergodic}, $\mu$ is continu and thus ${\bf 1}_A\mu-m_A\mu$
is the sum of continuous plus jumps and therefore its iterates satisfy Definition
\ref{def:mix-densityclass}. Proposition \ref{prop:mix-decay} now implies the theorem. \QED

\vspace{\fill}
\end{document}